\documentclass{article}

\usepackage{csquotes}
\usepackage[shortlabels]{enumitem}			
\usepackage{amsmath, amssymb, amsfonts}	
\usepackage[dvipsnames]{xcolor}
\usepackage[colorlinks]{hyperref}
\hypersetup{
	linkcolor=NavyBlue,
	citecolor=NavyBlue,
	filecolor=Red,
	urlcolor=NavyBlue,
	menucolor=Red,
	runcolor=Red}
 
\usepackage{biblatex}
\usepackage{tikz-cd}

\usepackage{amsthm}
\usepackage{multirow}
\newtheorem{thm}{Theorem}

\newtheorem{sol*}{Solution}
\theoremstyle{definition}

\newcommand{\pa}[1]{{\left (#1\right )} }
\newcommand{\mbar}[1]{\overline{\mathcal{M}}_{0,#1}}
\newcommand{\rk}{{\mathrm{rk}\,}}

\usepackage{lipsum}

\title{Improved Bounds on Ultra-Log Concavity of the Grothendieck Class of $\mbar n$}
\author{Eduardo Nascimento}
\date{June 2025}

\begin{document}
\maketitle
\begin{abstract}
    The class of the fine moduli space of stable $n$-pointed curves of genus zero,
    $\mbar{n}$, in the Grothendieck ring of varieties encodes its Poincaré polynomial.
    Aluffi-Chen-Marcolli conjecture that the Grothendieck class of $\mbar{n}$
    is real-rooted (and hence ultra-log-concave),
    and they proved an asymptotic ultra-log-concavity result for these polynomials. 
    
    We build upon their work, by providing effectively computable bounds for the error term
    in their asymptotic formula for $\rk H^{2l}(\mbar n)$.
    As a consequence, we prove that in the range $l \le \frac{n}{10\log n}$, the ultra-log-concavity
    inequality
    \[\left(\frac{\rk H^{2(l-1)}({\mbar{n}})}{\binom{n-3}{l-1}}\right)^2 \ge \frac{\rk H^{2(l-2)}({\mbar{n}})\rk H^{2l}({\mbar{n}})}{\binom{n-3}{l-2}\binom{n-3}{l}} \]
    holds for $n$ sufficiently large.
      
\end{abstract}

\section{Introduction}

Let $\mathbb L$ be the Lefschetz-Tate class (that of the affine line) in the Grothendieck ring of varieties.
Then, the class $[\mbar n]$ of the moduli space of stable $n$-pointed curves of genus zero
is a polynomial in $\mathbb L$ whose coefficients are the even betti numbers of $\mbar n$. 

Since $\mbar n$ is a smooth projective variety, by the Hard Lefschetz theorem its even 
betti numbers form a unimodal sequence.
In \cite{ACM}, Aluffi-Chen-Marcolli conjecture the much stronger property that its Grothendieck class
in the Grothendieck ring of varieties is real-rooted and thus, by a theorem of Newton, is ultra-log-concave.

From the work of Keel \cite{Kee92}, Getzler \cite{Get95}, and Manin \cite{Man95}, we have recursive
    formulas for the Grothendieck class of $\mbar n$ and differential and functional equations for
    a generating function of these polynomials.
Using these tools, Aluffi-Chen-Marcolli showed the following:

\begin{thm}[\cite{ACM}, Theorem 1.1]\label{thm:ACMulc}
For every $l$ there exists $N=N(l)$ such that for all $n>N$,
    \[\frac{\rk H^{2(l-2)}({\mbar{n}})}{\binom{n-3}{l-2}}\frac{\rk  H^{2l}({\mbar{n}})}{\binom{n-3}{l}} \le\pa{\frac{\rk H^{2(l-1)}({\mbar{n}})}{\binom{n-3}{l-1}}}^2.\]
\end{thm}

The key to their proof is the following asymptotic formula:
\begin{thm}[\cite{ACM}, Theorem 1.3]\label{thm:ACMasymptotic}
    For a fixed $l\ge 0$,
    \[\rk H^{2l}(\mbar n)\sim \frac{(l+1)^{l+n-1}}{(l+1)!}\]
    as $n\to \infty$.
\end{thm}
    
    More recently, in \cite{AMN}, Aluffi-Marcolli-N. found explicit formulas for these Grothendieck classes, such as the following:
    \begin{thm}[\cite{AMN}, Theorem 1.1]
        \[[\mbar{n}]= (1-\mathbb L)^{n-1} \sum_{k>0}\sum_{j>0} s(k+n-1,k+n-1-j) S(k+n-1-j,k+1) \mathbb L^{k+j}.\]
    \end{thm}
Here, $s(n,k)$ and $S(n,k)$ are the Stirling numbers of the first and second kind, respectively.
This formula was reproved by Eur-Ferroni-Matherne-Pagaria-Vecchi \cite{EFMPV25} with a different method.
They made use of the fact that the cohomology ring of $\mbar n$ is the Chow ring of a braid matroid with 
respect to its minimal building set to provide a combinatorial proof of this formula.

For our purposes, we will use the following explicit formula for the betti numbers of $\mbar n$:

\begin{thm}[\cite{AMN}, Corollary 1.5]
    Let, for $i>0$, 
    $$C_{nki}=\frac{(-1)^i(2ki+ni+k+n-1)+k-i}{i(i+1)}-\frac{1}{i(k+1)^i}\sum_{j=0}^{k+n-2}j^i$$

    Then,
    \begin{equation}\label{eq:AMNformula}
        \text{\emph{rk} } H^{2l}({\mbar{n}})=\sum_{k=0}^l \frac{(k+1)^{k+n-1}}{(k+1)!}\sum_{m=0}^{l-k}\frac{1}{m!}\sum_{i_1+\dots+i_m=l-k} C_{nki_1}\dots C_{nki_m}
    \end{equation}
    where the inner sum is over $m$-tuples of positive integers and the empty product is by convention $1$ (in particular the $m=0$ term is always $1$).
\end{thm}

Despite the formula above being somewhat cumbersome, 
we can still use it together with some crude estimates to improve
the asymptotic ultra-log concavity results from \cite{ACM}. 
For this, we find a refinement for Theorem \ref{thm:ACMasymptotic} by bounding the remaining terms.
Our main result is the following:

\begin{thm}\label{thm:improved lc}
   There exists an effectively computable constant $N>0$ such that for all integers $n>N$, and $l$ such that $2\le l\le \frac{n}{10\log n}$,
    \[\frac{\rk H^{2(l-2)}({\mbar{n}})}{\binom{n-3}{l-2}}\frac{\rk  H^{2l}({\mbar{n}})}{\binom{n-3}{l}} \le\pa{\frac{\rk H^{2(l-1)}({\mbar{n}})}{\binom{n-3}{l-1}}}^2.\]
\end{thm}
To prove this, we need the following result, which is a stronger version of Theorem \ref{thm:ACMasymptotic}:

\begin{thm}\label{thm:refined asymptotic}
    On the asymptotic where $n\to \infty$ with $l\le \frac{n}{10\log n}$,
\[\text{{rk} } H^{2l}({\mbar{n}})=\frac{(l+1)^{l+n-1}}{(l+1)!}\exp\left(O\left (\frac{1}{n^2}\right)\right ).\]

\end{thm}

Here we are using Landau's big-oh notation in the sense that
\[
f(x_1,x_2,\dots,x_m) = O(g(x_1,\dots,x_m))
\]
if there exist $C,N$ effectively computable such that $f \le Cg$ for all $(x_1,\dots,x_m)$ in the domain of $f$ (which is contained in the domain of $g$) with $\max x_i \ge N$.

Improving the range of $l$ for which Theorem \ref{thm:improved lc} holds to $l\le 1+n/2$ would show that for $n$ sufficiently large,
the Grothendieck classes of ${\mbar{n}}$ are ultra-log-concave, which would be evidence in favor of the conjecture of Aluffi-Chen-Marcolli \cite{ACM} that these classes are real-rooted.

\textit{Acknowledgments.} The author gratefully acknowledges support from Caltech's Summer Undergraduate Research Fellowship, made possible by a generous gift from Marcella Bonsal.

\section{Proofs}

\begin{proof}[Proof of Theorem \ref{thm:refined asymptotic}]

    By the integral bound $$\sum_{j=0}^{k+n-2} j^i \le \int_0^{k+n-1}x^i dx = \frac{(k+n-1)^{i+1}}{i+1}, $$
    we get that 
    \[|C_{nki}| \le \left|\frac{(-1)^i(2ki+ni+k+n-1)+k-i}{i(i+1)}\right|+\frac{(k+n-1)^{i+1}}{i(i+1)(k+1)^i}\]
    \[\le O(n)+\frac{(2n)^{i+1}}{i(i+1)(k+1)^i}\]
    Hence, 
    \[ C_{nki} = O\left (n^{i+1}\right )\]
    So we get a constant $C$ (which we assume without loss of generality is greater than one) such that 
    \[|C_{nki}|\le Cn^{i+1}\]
    Hence, for any $m,l,$
    \[
    \left\lvert\sum_{i_1+\dots+i_m=l-k} C_{nki_1}\dots C_{nki_m}\right\rvert
    \le \sum_{i_1+\dots+i_m=l-k} C^m n^{(i_1+1)+\cdots+(i_m+1)}.
    \]
    Note that the term inside the sum does not actually depend on the $i_j$'s.
    The number of $m$-tuples of positive integers with sum $l-k$ is
    $$\binom{l-k-1}{m-1}\le 2^{l-k+m-1}\le 4^{l-k},$$

    and so we conclude that
    \[
    \left\lvert\sum_{i_1+\dots+i_m=l-k} C_{nki_1}\dots C_{nki_m}\right\rvert
    \le C^m n^{l-k+m}4^{l-k}.
    \]
    Now, let's bound the sum over $m$ by the number of terms times the largest term.
    Since $m\le l-k$ and $m!\ge 1$,
    \[
    \left\lvert\sum_{m=0}^{l-k}\frac{1}{m!}\sum_{i_1+\dots+i_m=l-k} C_{nki_1}\dots C_{nki_m}\right\rvert
    \le (l-k+1)\left(4Cn^2 \right)^{l-k}.
    \]
    We are now ready to bound the desired error term, once again using a maximum times number of terms bound:
\begin{align*}
    & \left\lvert \rk H^{2l}(\mbar n) - \frac{(l+1)^{l+n-1}}{(l+1)!} \right\rvert= \\
    &\qquad = \left\lvert 
        \sum_{k=0}^{l-1} \frac{(k+1)^{k+n-1}}{(k+1)!}
        \sum_{m=0}^{l-k} \frac{1}{m!} 
        \sum_{\substack{i_1+\dots+i_m = l-k \\ i_j \ge 0}} 
        C_{nki_1} \dots C_{nki_m} 
    \right\rvert
\end{align*}
\[\le l \cdot \max_{0\le k \le l-1} \frac{(k+1)^{k+n-1}}{(k+1)!} \cdot \left ((l-k+1)\left(4Cn^2\right)^{l-k}\right ). \]

We take outside the maximum terms which won't matter to the final estimates:
\[\le l(l+1)\max_{0\le k \le l-1} \frac{(k+1)^{k+n-1}}{(k+1)!} \left(4Cn^2\right)^{l-k}.\]

We multiply and divide by the claimed main term:
\[\le \frac{(l+1)^{l+n-1}}{(l+1)!} \cdot l(l+1)  \max_{0\le k \le l-1} \frac{(k+1)^{k+n-1}}{(k+1)!}\frac{(l+1)!}{(l+1)^{l+n-1}} \left(4Cn^2\right)^{l-k}.\]

Now using Robbins' explicit version of Stirling's formula \cite{Rob55}:

\[\frac{(k+1)^{k+n-1}}{(k+1)!}\frac{(l+1)!}{(l+1)^{l+n-1}} \le\]
\[\le  \dfrac{(k+1)^{k+n-1}}{\sqrt{2\pi (k+1)}\left(\frac{k+1}{e}\right)^{k+1}e^{\frac{1}{12(k+1)+1}}}\frac{\sqrt{2\pi (l+1)}\left(\frac{l+1}{e}\right)^{l+1}e^{\frac{1}{12(l+1)}}}{(l+1)^{l+n-1}}\]
Joining terms, this is
\[= \frac{e^{l+1}}{e^{k+1}}\frac{(k+1)^{n-2}}{(l+1)^{n-2}}\frac{\sqrt{l+1}}{\sqrt{k+1}}e^{\frac{1}{12(l+1)}-\frac{1}{12(k+1)+1}}\]
\[= \left(\frac{k+1}{l+1}\right)^{n-2.5}e^{l-k+\frac{1}{12(l+1)}-\frac{1}{12(k+1)+1}}\]
\[\le \left(\frac{k+1}{l+1}\right)^{n-2.5}e^{l-k}\]

With this, we bound the error term by
\[\le \frac{(l+1)^{l+n-1}}{(l+1)!} \cdot l(l+1) \max_{0\le k \le l-1} \left(\frac{k+1}{l+1}\right)^{n-2.5} \left(4eCn^2\right)^{l-k}.\]
Taking some terms out of the maximum again, 
\[\le \frac{(l+1)^{l+n-1}}{(l+1)!} \cdot l(l+1)  (l+1)^{2.5} \max_{0\le k \le l-1} \left(\frac{k+1}{l+1}\right)^{n} \left(4eCn^2\right)^{l-k}\]

But now note that
\[\left ( \frac{k+1}{l+1} \right )^n \left(4eCn^2\right)^{l-k}  = \left ( 1- \frac{l-k}{l+1} \right )^n \left(4eCn^2\right)^{l-k}\]
\[\le \exp\left(-(l-k)\frac{n}{l+1} + (l-k)\log(4eCn^2)\right).\]
For $l \le \frac{n}{10\log(n)}$ (and $n$ sufficiently large), this is increasing in $k$, so its maximum is attained at $k=l-1$ and is at most
\[\exp\left (-\frac{n}{l+1}+\log(4eCn^2)\right ).\]
Hence,

\[\left\lvert \rk H^{2l}(\mbar n) - \frac{(l+1)^{l+n-1}}{(l+1)!} \right\rvert \le \]
\[\le \frac{(l+1)^{l+n-1}}{(l+1)!} \cdot l(l+1)  (l+1)^{2.5} \exp\left (-\frac{n}{l+1}+\log(4eCn^2)\right )\]
\[\le \frac{(l+1)^{l+n-1}}{(l+1)!}\exp\left(-\frac{n}{l+1}+6.5\log n + O(1)\right)\]
Substituting our inequality $l \le \frac{n}{10\log(n)}$, we get
\[\le \frac{(l+1)^{l+n-1}}{(l+1)!}\exp\left(-2\log n + O(1)\right)\]
\[\le \frac{(l+1)^{l+n-1}}{(l+1)!}O\left(\frac{1}{n^2}\right)\]

This gives us 
\[\text{{rk} } H^{2l}({\mbar{n}})=\frac{(l+1)^{l+n-1}}{(l+1)!}\left(1+O\left(\frac{1}{n^2}\right)\right)=\frac{(l+1)^{l+n-1}}{(l+1)!}\exp\left(O\left (\frac{1}{n^2}\right)\right)\]
\end{proof}

\begin{proof}[Proof of Theorem \ref{thm:improved lc}]
    
By Theorem \ref{thm:refined asymptotic}, we have that
\[\frac{(\text{\emph{rk} } H^{2(l-1)}({\mbar{n}}))^2}{\text{\emph{rk} } H^{2(l-2)}({\mbar{n}})\text{\emph{rk} } H^{2l}({\mbar{n}})}=\frac{\frac{l^{2l+2n-4}}{l!^2}}{\frac{(l+1)^{l+n-1}}{(l+1)!}\frac{(l-1)^{l+n-3}}{(l-1)!}}\exp\left(O\left(\frac{1}{n^2}\right)\right)\]
\[=\frac{(l+1)!(l-1)!}{l!^2}\frac{l-1}{l+1}\frac{(l^2)^{l+n-2}}{(l+1)^{l+n-2}(l-1)^{l+n-2}}\exp\left (O\left(\frac{1}{n^2}\right)\right )\]

\[=\frac{l-1}{l} \left ( \frac{l^2}{l^2-1} \right )^{l+n-2}  \exp\left (O\left(\frac{1}{n^2}\right)\right ).\]

By the inequality $1+x\le e^x$, applied to $x=-\frac{1}{l^2}$, we get that

\[\left ( \frac{l^2}{l^2-1} \right )^{l+n-2} = \left ( \frac{1}{1-\frac{1}{l^2}} \right )^{l+n-2} \ge \left ( e^{\frac{1}{l^2}} \right )^{l+n-2} = \exp\left(\frac{l+n-2}{l^2}\right).\]
By the same inequality, applied to $x=\frac{1}{l-1}$ (and some rearranging), we get that
\[\frac{l-1}{l} \ge e^{-\frac{1}{l-1}}.\]
Hence, 
\[\frac{l-1}{l} \left ( \frac{l^2}{l^2-1} \right )^{l+n-2}  \exp\left (O\left(\frac{1}{n^2}\right)\right ) \ge \]
\[\ge \exp\left(-\frac{1}{l-1} + \frac{l+n-2}{l^2} + O\left(\frac{1}{n^2}\right)\right)=\]
\[= \exp\left(-\frac{1}{l(l-1)}-\frac{2}{l^2} + \frac{n}{l^2} + O\left(\frac{1}{n^2}\right)\right)\]
\[= \exp\left(\frac{n}{l^2} + O\left(\frac{1}{l^2}\right)\right).\]

But note also that 

\[\frac{\binom{n-3}{l-2}\binom{n-3}{l}}{\binom{n-3}{l-1}^2} = \frac{l-1}{l}\frac{n-l-2}{n-l-1} \]
\[=\left(1-\frac{1}{l}\right)\left(1-\frac{1}{n-l-1}\right) = \exp\left (O\pa{\frac{1}{l}}\right )\]
because $l\le \frac{n}{10\log n}$ and so $n-l \ge n/2$ for $n$ large.
Hence,
\[\frac{(\rk H^{2(l-1)}({\mbar{n}})/\binom{n-3}{l-1})^2}{\rk H^{2(l-2)}({\mbar{n}})\rk H^{2l}({\mbar{n}})/\binom{n-3}{l-2}\binom{n-3}{l}} \ge
 \exp\left(\frac{n}{l^2} + O\left(\frac{1}{l}\right)\right)>1\]
for $n$ sufficiently large.
\end{proof}

\defbibheading{bibliography}{%
    \section{Bibliography}}
\printbibliography

\end{document}